\documentclass[12pt, reqno]{amsart}
\usepackage{amsmath, amsthm, amscd, amsfonts, amssymb, graphicx, color}
\usepackage[bookmarksnumbered, colorlinks, plainpages]{hyperref}

\newtheorem{thm}{Theorem}[section]
 
 \newtheorem{corol}[thm]{Corollary}
 
 \newtheorem{prop}[thm]{Proposition}
 \newtheorem{defn}[thm]{Definition}
 \newtheorem{rem}[thm]{Remark}
 \newtheorem{examp}[thm]{Example}

 \numberwithin{equation}{section}

\theoremstyle{remark}

\numberwithin{equation}{section}

\begin{document}
\setcounter{page}{1}

\title[STABLE NORMED ALGEBRAS]{stable normed algebras}

\author[E. Ansari, S. Nouri]{E. Ansari Piri$^{*}$, S. Nouri }

\address{Department of Pure Mathematics,   Faculty of
Mathematical Science,
  University of Guilan,   Rasht,   Iran.}
\email{\textcolor[rgb]{0.00,0.00,0.84}{eansaripiri@gmail.com; solmaznouri85@gmail.com}}


\date{Received: xxxxxx; Revised: yyyyyy; Accepted: zzzzzz.
\newline \indent $^{*}$ Corresponding author
\newline \indent 2010 AMS Mathematics Subject Classification: 46H99}

\begin{abstract} In this paper we introduce a new property for normed
algebras. This property which we call it stability, plays a key role in the studying
of the theory of almost multiplier maps. In this note we study some of
the basic properties of this notion and characterize stable algebras among similar algebras.
\end{abstract}

 \maketitle

\section{Introduction and preliminaries}

There are many papers discussing on the stability of maps
on normed algebras. In this note we do not discuss on the
stability of such maps on a normed algebra but, we introduce a
special property for a class of normed algebras which we call stability on normed
algebras and these algebras are called stable normed algebras.
The general theory of multipliers on a faithful Banach
algebra was introduced by Wang in \cite{W} and Birtal in \cite{Bir}.
Nearly all of the results about multiplier maps are proved for faithful normed
algebras. In \cite{A-N}, we have introduced the almost multiplier maps and investigate
some properties of these maps for stable normed algebras instead of faithful normed algebras.

In section \ref{section2} of this note, we define stable normed algebras and compare this group of algebras
with the faithful algebras and the normed algebras with an approximate identity. In section \ref{section3},
we prove a result for stable normed algebras which is in fact an equivalent definition
of stable normed algebras and also we define the notion of stable normed algebras on the modules.

In section \ref{section4}, we close the paper by introducing some questions which we believe to be open.

We recall that a normed algebra $A$ is left (resp. right)
faithful if for all $x\in A, ~ xA=0 ~(resp.~Ax=0)$, implies $x=0$.
The normed algebra $A$ is faithful if it is left and right
faithful. Let $A$ be a normed algebra. A net $( e_{\lambda})$ in
$A$ is called a left ( resp. right, two-sided ) approximate
identity, if for all
 $x\in A$, $\lim_{\lambda \in \Lambda} e_{\lambda} x= x$,
  (resp. $\lim_{\lambda \in \Lambda} x e_{\lambda} = x$,
  $\lim_{\lambda \in \Lambda} e_{\lambda} x= x=\lim_{\lambda \in \Lambda} x e_{\lambda}$).

 \vspace{0.5 cm}
\section{characterizing stable normed algebras}\label{section2}

In this section we introduce the notion of stable normed algebra
and compare it with similar algebras. In fact we prove that every
stable normed algebra is faithful and give an example of a
faithful Banach algebra which is not stable. Also we compare
these algebras with the algebras having an approximate identity.

 \vspace{0.5 cm}
\begin{defn}
We say that the normed algebra $A$ is a left (resp. right) stable
normed algebra if for all $a\in A$ and $M>0$; if we have
$\|ab\|\leqslant M \|b\|~$ $( resp. ~\|ba\|\leqslant M \|b\|~)$
for all $b\in A$, then we can conclude $\|a\|\leqslant M$. The
normed algebra $A$ is stable if it is both left and right stable.
\end{defn}

\begin{prop}
Every stable normed algebra $A$ is faithful.
\end{prop}

\begin{proof} For $a\in A$, let $aA=0$ and $\varepsilon >0$. For
all $b\in A$ we have $\|ab\|=0\leqslant \varepsilon \|b\|$ and so
$\|a\|\leqslant \varepsilon$, which implies $a=0$.
\end{proof}
 \vspace{0.5 cm}
Now obviously, all theorems and results which hold for faithful normed
algebras also valid for stable normed algebras.

Here we give an example to show that the class of stable normed
algebras are essentially different from the class of faithful
algebras.

\begin{examp}
Suppose $A=( {\ell}^1 (\Bbb N), \|.\|_1)$ with the usual
definition of operations (pointwise). Then $A$ is a faithful
Banach algebra which is not stable. For if define
$a_n=\frac{1}{2^n}$, then $(a_n)\in A$, and for all $(b_n) \in A$
we have $\|(a_n). (b_n) \|\leq \frac{1}{2} \|(b_n)\|$, but
$\|(a_n)\|=1 > \frac{1}{2}$.
\end{examp}

\vspace{0.5cm}
 Finally, we give an example of a normed algebra
witch is neither a faithful nor a stable algebra.

\begin{examp}
The matrix algebra $M_{2\times 2}(\Bbb C)$ is an stable normed
algebra, but it's subalgebra $A=\{ \begin{pmatrix} a&b\\0&0
\end{pmatrix}: a,b \in \Bbb C\}$, with the usual matrix multiplication and
norm $\|\begin{pmatrix} a&b\\0&0
\end{pmatrix}\|=max \{|a|, |b| \}$, is a Banach algebra which is not stable and faithful.
Since

$\|\begin{pmatrix} a&b\\0&0
\end{pmatrix}~\begin{pmatrix}
c&d\\0&0
\end{pmatrix}\|=\|\begin{pmatrix}
ac&ad\\0&0
\end{pmatrix}\|=|a|~max\{|c|,|d|\}= |a|~\|\begin{pmatrix}
c&d\\0&0
\end{pmatrix}\|$
But it is not necessarily true that, $\|\begin{pmatrix} a&b\\0&0
\end{pmatrix}\|\leq |a|$.

Moreover, since $\begin{pmatrix} a&b\\0&0
\end{pmatrix} \begin{pmatrix}
c&d\\0&0
\end{pmatrix}=\begin{pmatrix}
ac&ad\\0&0
\end{pmatrix}$, if $\begin{pmatrix} a&b\\0&0
\end{pmatrix} A=0$
it is not necessarily true that, $\begin{pmatrix} a&b\\0&0
\end{pmatrix}=0$.

If we define $\|\begin{pmatrix} a&b\\0&0
\end{pmatrix}\|= \sqrt{|a|^2 + |b|^2}$ or $\|\begin{pmatrix}
a&b\\0&0
\end{pmatrix}\|=|a|+|b|$ on $A$. Then $A$ is not
stable yet.
\end{examp}

\vspace{0.5cm}
\begin{rem}
The above example also shows that a subalgebra of an stable
normed algebra is not necessarily stable.
\end{rem}

 \vspace{0.5 cm}
It is easy to see that if $A$ has a bounded approximate identity, then it is faithful.
Here we verify the relation between a stable normed algebra and the algebra with a bounded
approximate identity.

\begin{prop}\label{prop1}
Every normed algebra with a bounded left (resp. right)
approximate identity with bound 1, is a left (resp. right) stable
normed algebra.
\end{prop}

\begin{proof} Suppose $A$ is a normed algebra and let $(e_\alpha)_{\alpha \in I}$ be a bounded left
approximate identity of A with bound 1. For $a\in A$ and $M>0$ if we
 have $\|ab\|\leqslant M \|b\|$, for all $b\in A$, then
  for all $\alpha \in I$, $\|a e_{\alpha}\|\leqslant M
\|e_{\alpha}\|\leqslant M$ and therefore $\|a\|\leqslant M$.
\end{proof}
\begin{corol}\label{corol1}
Every unital normed algebra is stable.
\end{corol}

\begin{prop}
Let $(A,\|.\|)$ be a commutative normed algebra with a bounded sequential approximate
identity. Then there is an equivalent algebra norm $|.|$ on $A$ such that, $(A,|.|)$ is
stable.
\end{prop}

\begin{proof} By \cite[theorem 7.3]{D}, there is an equivalent norm $|.|$ on $A$ for which
there exists another sequential approximate identity on $A$ with
bound 1. So by proposition \ref{prop1}, $(A, |.|)$ is stable.
\end{proof}

\begin{prop}\label{prop2}
Every C*-algebra $A$ is Stable.
\end{prop}

\begin{proof} For $a\in A$ and all $M>0$, if we have $\|ab\|\leq M\|b\|$, since $A$ is C*- algebra
we have $\|a\|=\sup_{\|b\|\leq 1} \|ab\|\leq M$.
\end{proof}
There are of course unital stable normed algebras which are not
C*-algebra.

 \vspace{0.5 cm}
By proposition \ref{prop1} every normed algebra with a bounded approximate identity of
bound 1 is stable, but proposition \ref{prop2} shows that for $ C^*$- algebras, the bound 1 is not necessarily
for $(e_\alpha)_{\alpha \in I}$ making $A$ to a stable normed algebra. Since C*-algebras
may have a bounded approximate identity with bound greater than 1 and, the following theorem shows that
a stable Banach algebra may have no approximate identity.

\vspace{0.25cm}
\begin{thm}\label{thm3}
There are stable Banach algebras with no approximate identity.
\end{thm}

\begin{proof}
Let $B$ be any unital commutative Banach algebra, and suppose that $a\in B$,
 with $\|a\|=1$ where $a$ is not an invertible element.
 Moreover suppose $$\inf \{\|ax\|: \|x\|=1\}=1.$$
 
Define $A=a B$. Then $A$ with the norm $\|a b\|~'=\|b\|$ for $b\in B$, is
a commutative Banach algebra. Now we show that $A$ is a stable
 Banach algebra without any approximate identity.

Assume $A$ has an approximate identity $(a e_{\lambda})_{\lambda \in \Lambda}$. Then there exists $\lambda_0$ such that
 $\| a (a{e_{\lambda_{0}}})- a\|~'<1$ and so $\|a {e_{\lambda_{0}}} -1\|<1$.
 Thus $(a{e_{\lambda_{0}}})^{-1}$ exists in $B$ and $a{e_{\lambda_{0}}}(a{e_{\lambda_{0}}})^{-1}=1$.
  It means $a$ has inverse in $B$, which is a contradiction. So $A$ has no
 approximate identities.

 Now we show that $A$ is stable. Suppose $b\in B$ and $M>0$. If we have
 $\|(a b)( a x)\|~' \leq M \| a x\|~'$ for all $x\in B$. Then
 $\|bax \|\leq M \|x\|$. Since $B$ is unital, by corollary \ref{corol1}, it
  is stable. Therefore by definition of stable algebras $\|ba\|\leq M$.
  Now we have
   $$\|b\|=\|b\|.1 \leq \|b\|~\|\frac{b}{\|b\|} ~ a\|\leq M .$$ So $\|a b\|~'=\|b\|\leq M$.
\end{proof}

In the next example, we show that such an algebra $B$ asserted in theorem \ref{thm3}, exists.

 \begin{examp}
Let $D\subseteq \Bbb C$ be the open unit disc and suppose ${\mathcal{A}}(D)$ is the set of all
 continuous functions on $\overline{D}$ which are
holomorphic on $D$. Then ${\mathcal{A}}(D)$ with the usual definitions of operations (pointwise) and the
 uniform norm $\|.\|_{\infty}$, is a unital commutative Banach algebra.

Now assume $f_{0} (z)=z$, then $\|f_{0}\|_{\infty}=1$ and it is a singular element of ${\mathcal{A}}(D)$.
Also by the maximum modulus principle theorem $$\|f_{0} f\|_{\infty} = \sup_{z\in \overline{D}} |z f(z)|= \sup_{z\in \partial D} |z f(z)|=\|f\|_{\infty},$$
for all $f\in {\mathcal{A}}(D)$. This shows that $$\inf\{\|f_{0} f\|_{\infty}: \|f\|_{\infty}=1\}=\{\|f\|_{\infty}: \|f\|_{\infty}=1\}=1.$$
\end{examp}

 \vspace{0.5 cm}
\section{some results on stable normed algebras}\label{section3}

In \cite{A-N}, we have proved some results about almost
multiplier maps on the stable normed algebras and there are many
questions about validity of famous theorems for stable normed
algebras. In this section we present an equivalent relation for
stability and investigate some properties which are valid for
stable normed algebras. We also extend the notion of stable normed
algebras for modules.

 \vspace{0.5 cm}
\begin{thm}\label{thm5}
The normed algebra $A$ is stable if and only if $\|a\|=||L_a\|$,
for all $a\in A$.
\end{thm}

\begin{proof} Suppose $A$ is a stable normed algebra. By the
definition if, for all $a\in A$ and $M>0$; if we have
$\|ab\|\leqslant M \|b\|~$ for all $b\in A$, then we can conclude
$\|a\|\leqslant M$. So if for all $a\in A$ and $M>0$; if we have
$\|L_a\|\leqslant M$ then $\|a\|\leqslant M$. It means
$\|a\|\leqslant \|L_a\|$ and so $\|a\|=\|L_a\|$.

Now suppose $\|a\|=\|L_a\|$, for all $a\in A$ if for all $a\in A$ and
$M>0$; if we have $\|ab\|\leqslant M \|b\|~$ for all $b\in A$, then by taking
 spermium we have $\|a\|=\|L_a\|\leq M$.
\end{proof}

\begin{thm}
Suppose that $A$ is a stable Banach algebra then $A$ is isometrically
isomorphic with a subalgebra of $B(A)$.
\end{thm}

\begin{proof}
 Let $A$ be a stable normed algebra and $B(A)$ be the
Banach algebra of all bounded linear maps on $A$. Since $A$ is
faithful the map $\phi: A\longrightarrow B(A),\quad a\to L_a$, is
the natural embedding.

In the following, we can see that $\phi(A)$ is a closed
subalgebra of $\mathcal{B}(A)$, in norm topology. Suppose
$L_{a_n} \rightarrow T$ in $\phi(A)$, therefore $( L_{a_n})$ is
a Cauchy sequence. So for any arbitrary $\varepsilon>0$ and $b\in
A$ with norm less than 1, $\|a_n b- a_m b\|\leq \varepsilon$, for
some $N>0$ and every $m>n>N$. Since $A$ is stable, we have $\|a_n
-a_m \|\leq \varepsilon$. It means that there is an $a\in A$, such
that $a_n \rightarrow a$ and therefore $T=L_a$. Finally by theorem \ref{thm5},
$\phi$ is an isometry.
\end{proof}

\begin{prop}
Let $A$ be a stable normed algebra then its completion is stable
too.
\end{prop}

\begin{proof}
By theorem \ref{thm5}, it is enough to show that
$\|\tilde{a}\|=\|L_{\tilde{a}}\|$, for all $\tilde{a}\in
\tilde{A}$. Since $A$ is stable we have $\|a\|=\|L_a\|$, for all
$a\in A$. Let $\tilde{a}\in \tilde{A}$, so there is a sequence
$(a_n)\in A$ such that $a_n\rightarrow \tilde{a}$. Therefore
$\|L_{a_n}\|=\|a_n \| \rightarrow \| \tilde{a} \|$. Moreover,
since $B(\tilde{A})=\widetilde{B(A)}$, $L_{a_n}\rightarrow
L_{\tilde{a}}$ in $B(\tilde{A})$. So $\|L_{a_n}\|\rightarrow
\|L_{\tilde{a}}\|$, it means $\|\tilde{a}\|=\|L_{\tilde{a}}\|$.
\end{proof}

Now we extend the notion of stability on modules.

\begin{defn}
Let $A$ be a normed algebra and $X$ be a left(resp. right) normed $A$-module.
We say $A$ is left(resp. right) stable on $X$ if
for all $a\in A$ and $M>0$; if we have
$\|ax\|\leqslant M \|x\|~$ $( resp. ~\|xa\|\leqslant M \|x\|~)$
for all $x\in X$, then we can conclude $\|a\|\leqslant M$.
\end{defn}

\begin{prop}
Let $A$ be a stable normed algebra. Then

{\rm i)}
A is stable on every ideal $I$ of $A$.

{\rm ii)} If $A$ is unital, then $A$ is stable on $A^*$.

\end{prop}

\begin{proof}
(ii) Suppose for $a\in A$ and $M>0$, we have $\|fa\|\leq M \|f\|$, for all $f\in A^*$.
 By the Hahn Banach theorem, there exists $f\in A^*$ such that $|f(a)|=\|a\|$ and $\|f\|\leq 1$. So
 for this $f$ we have $$\|a\|=|f(a)|=|f(a.1)|\leq \|fa\|\leq M \|f\|\leq M.$$
\end{proof}

\begin{thm}
Let $A$ be a normed algebra and $X$ be a left normed $A$-module. If $A$ is stable on $X$ then

{\rm i)} If for $a\in A$ we define $\varphi_a:X\rightarrow X$ by $\varphi_a (x)=ax$, then $\varphi_a \in B(X)$
and $\|a\|\leq ||\varphi_a\|\leq k \|a\|$, where $k$ is the constant module coefficient.

{\rm ii)} The completion $\tilde{A}$ of $A$ is stable on $X$.

\end{thm}

\section{a list of open questions}\label{section4}
Here we give some questions which we believe they are open.

\begin{enumerate}
\item If $(A,\|.\|)$ is stable and $p\in En(A)$. Is $(A,p)$ stable?

\item Suppose $A$ is a stable normed algebra. Is the second dual
 $A^{**}$ of $A$ (with the first or second Arens product) stable?

\item If $N$ is a bi-ideal of stable normed algebra $A$, is $\frac{A}{N}$ stable?

 As is well known, the quotient algebra of a $C^*$-algebra is a $C^*$-algebra. But to
 prove this matter the fact that every $C^*$-algebra has an approximate identity has been used
 and therefore for a general stable normed algebra we think the answer of this question is negative.

\item If $I$ is an ideal of a stable algebra, is $I$ itself a stable normed algebra?

\end{enumerate}

\end{document}